\documentclass[11pt, final]{article}
\usepackage{a4}
\usepackage{amsmath}%
\usepackage{amstext}%
\usepackage{amssymb}%
\usepackage{showkeys}%
\usepackage{epsfig}%
\usepackage{cite}
\usepackage{tikz}
\usepackage{subfig}
\usepackage{caption}
\usepackage{multirow}
\usepackage{booktabs}
\usepackage{floatrow}

\usepackage{listings}
\usepackage{geometry}
\newtheorem{theorem}{Theorem}

\newtheorem{axiom}{Axiom}

\newtheorem{corollary}[theorem]{Corollary}

\newtheorem{definition}[axiom]{Definition}

\newtheorem{lemma}[theorem]{Lemma}

\newtheorem{rem}[theorem]{Remark}
\newenvironment{remark}{\rem\rm}{\endrem}

\newcounter{unnumber}

\newtheorem{prf}[unnumber]{Proof.}
\newenvironment{proof}{\prf\rm}{\hfill{$\blacksquare$}\endprf}
\newcommand{\R}{\mathbb{R}}%
\newcommand{\N}{\mathbb{N}}%
\newcommand{\ol}{\overline}%
\newcommand{\ul}{\underline}%

\renewcommand{\>}{\right\rangle}
\newcommand{\<}{\left\langle}

\DeclareMathOperator*\dom{dom}%
\DeclareMathOperator*\gr{Gr}%
\DeclareMathOperator*\id{Id}%
\DeclareMathOperator*\prox{prox}%
\DeclareMathOperator*\argmin{argmin}

\DeclareMathOperator*\zer{Zer}
\DeclareMathOperator*\loc{loc}
\DeclareMathOperator*\fix{Fix}

\title{Second order forward-backward dynamical systems for monotone inclusion problems}

\author{Radu Ioan Bo\c{t} \thanks{University of Vienna, Faculty of Mathematics, Oskar-Morgenstern-Platz 1, A-1090 Vienna, Austria,
email: radu.bot@univie.ac.at.} \and
Ern\"{o} Robert Csetnek \thanks {University of Vienna, Faculty of Mathematics, Oskar-Morgenstern-Platz 1, A-1090 Vienna, Austria,
email: ernoe.robert.csetnek@univie.ac.at. Research supported by FWF (Austrian Science Fund), Lise Meitner Programme, project M 1682-N25.}}

\begin{document}
\maketitle

\noindent \textbf{Abstract.} We begin by considering second order dynamical systems of the from $\ddot x(t) + \gamma(t)\dot x(t) + \lambda(t)B(x(t))=0$, 
where $B: {\cal H}\rightarrow{\cal H}$ is a cocoercive operator defined on a real Hilbert space ${\cal H}$, 
$\lambda:[0,+\infty)\rightarrow [0,+\infty)$ is a relaxation function  and $\gamma:[0,+\infty)\rightarrow [0,+\infty)$ a damping function, both depending on time. For the generated trajectories, we show existence and uniqueness of 
the generated trajectories as well as their weak asymptotic convergence to a zero of the operator $B$.
The framework allows to address from similar perspectives second order dynamical systems associated with the problem of finding zeros of the sum of a maximally 
monotone operator and a cocoercive one. This captures as particular case the minimization of the sum of a nonsmooth 
convex function with a smooth convex one. Furthermore, we prove that when $B$ is the gradient of a smooth convex function the value of the latter 
converges along the ergodic trajectory to its minimal value with a rate of ${\cal O}(1/t)$. \vspace{1ex}

 \vspace{1ex}


\section{Introduction and preliminaries}\label{sec-intr}

This paper is motivated by the heavy ball with friction dynamical system
\begin{equation}\label{heavy-ball}\ddot x + \gamma\dot x + \nabla f(x) = 0,\end{equation} which is a nonlinear oscillator with damping $\gamma > 0$ and potential $f:{\cal H}\rightarrow\R$, 
supposed to be a convex and differentiable function defined on the real Hilbert space ${\cal H}$. The system \eqref{heavy-ball} is a simplified version of the differential system describing the motion of 
a heavy ball that keeps rolling over the graph of the function $f$ under its own inertia until friction stops it at a critical point of $f$ 
(see \cite{att-g-r}). Motivated by different models of friction, in \cite{aac, cabot} a generalized version of \eqref{heavy-ball} has been investigated in finite-dimensional spaces, by replacing the damping $\gamma \dot x$ with $\partial \Phi(\dot x)$, which is
the convex subdifferential of a convex function $\Phi : \R^n \rightarrow \R$ at $\dot x$.

The second order dynamical system \eqref{heavy-ball} has been considered by several authors in the context of minimizing the function $f$, 
these  investigations being either concerned with the asymptotic convergence of the generated trajectories to a critical point of $f$ or with the 
convergence of the function value along the trajectories to its global minimum value (see \cite{alvarez2000, att-g-r, antipin, att-alv}). It is also worth 
to mention that the time discretization of the heavy ball with friction dynamical system leads to so-called inertial type algorithms, 
which are numerical schemes sharing the feature that the current iterate of the generated sequence is defined by making use of the previous two iterates (see, for instance, 
\cite{alvarez2000, alv-att-sva, alvarez2004, b-c-h-inertial, moudafi-oliny2003, pp}). 

In order to approach the minimization of $f$ over a nonempty, convex and closed set $C \subseteq {\cal H}$, the gradient-projection second order dynamical system 
\begin{equation}\label{gr-pr}\ddot x + \gamma\dot x + x-P_C(x-\eta\nabla f (x)) = 0\end{equation}
has been considered, where  $P_C : {\cal H} \rightarrow C$ denotes the projection onto the set $C$ and $\eta>0$. Convergence statements for the trajectories to
a global minimizer of $f$ over $C$ have been provided in  \cite{att-alv, antipin}. Furthermore, in \cite{att-alv}, these investigations have been expanded to more general 
second order dynamical systems of the form
\begin{equation}\label{alv-att-nonexp-intr}\ddot x + \gamma\dot x + x-Tx = 0,\end{equation}
where $T:{\cal H} \rightarrow{\cal H}$ is a nonexpansive operator. It has been shown that when $\gamma^2 > 2$ the trajectory of \eqref{alv-att-nonexp} converges weakly to an element in the 
fixed points set of $T$, provided the latter is nonempty.

In this manuscript, we first treat the second order dynamical system 
\begin{equation}\label{paper-coc}\ddot x(t) + \gamma(t) \dot x(t) + \lambda(t)B(x(t))=0,\end{equation}
where $B: {\cal H}\rightarrow{\cal H}$ is a cocoercive operator, 
$\lambda:[0,+\infty)\rightarrow [0,+\infty)$ is a relaxation function in time and $\gamma:[0,+\infty)\rightarrow [0,+\infty)$ is a continuous damping parameter. 
We refer the reader to  
\cite{su-boyd-candes, att-ch, att-p-r, cabot-engler-gadat, cabot-engler-gadat-tams} for other works where second order 
differential equations with time dependent damping  
have been considered and investigated in connection with optimization problems. 
On the other hand, second order dynamical systems governed by cocoercive operators have been recently considered also in \cite{att-maing}, however, with constant relaxation and 
damping functions.  The existence and uniqueness of strong global solutions for \eqref{paper-coc} is obtained by applying the classical Cauchy-Lipschitz-Picard Theorem 
(see \cite{haraux}). We also show that under mild assumptions on the relaxation function the trajectory $x(t)$ converges weakly as $t \rightarrow +\infty$ 
to a zero of the operator $B$, provided the latter has a nonempty set of zeros. To this end we use the continuous version of the Opial Lemma 
(see also \cite{att-alv, antipin, alvarez2000}, where similar techniques have been used). 

Further, we approach the problem of finding a zero of the sum of a maximally monotone operator and a cocoercive one via a second order dynamical system formulated by 
making use of the resolvent of the set-valued operator, see \eqref{dyn-syst-fb}. 
Dynamical systems of implicit type have been already considered in the literature in 
\cite{abbas-att-arx14, bolte-2003, b-c-dyn-KM, b-c-dyn-pen, att-sv2011, abbas-att-sv, att-alv-sv}. We specialize these investigations to the minimization of the sum of a  
nonsmooth convex function with a smooth convex function, which is approached by means of a second order dynamical system of forward-backward type. This fact allows us to recover and improve 
results given in \cite{att-alv, antipin} in the context of studying \eqref{gr-pr}. We also emphasize the fact that the explicit discretization of the second order forward-backward dynamical system gives rise
to a relaxed forward-backward algorithm with inertial effects. By approaching minimization problems from continuous perspective we expect to gain more insights into how to properly chose the relaxation and damping parameters in the corresponding iterative schemes in order to 
improve their convergence behavior. Finally, whenever $B$ is the gradient of a smooth convex function we
show that the value of the latter converges along the ergodic trajectories generated by \eqref{paper-coc}  to its minimum value with a rate of convergence of  ${\cal O}(1/t)$. 
                                                       
Throughout this paper $\N= \{0,1,2,...\}$ denotes the set of nonnegative integers and ${\cal H}$ a real Hilbert space with inner product
$\langle\cdot,\cdot\rangle$ and corresponding norm $\|\cdot\|=\sqrt{\langle \cdot,\cdot\rangle}$.

\section{Existence and uniqueness of strong global solutions}\label{sec2}

This section is devoted to the study of existence and uniqueness of strong global solutions of a second order dynamical 
system governed by Lipschitz continuous operators. 

Let $B:{\cal H}\rightarrow {\cal H}$ be an $L$-Lipschitz continuous operator (that is $L \geq 0$ and $\|B x-B y\|
\leq L \|x-y\|$ for all $x,y\in{\cal H}$), $\lambda,\gamma:[0,+\infty)\rightarrow [0,+\infty)$ be Lebesgue measurable functions, 
$u_0,v_0\in {\cal H}$ and consider the dynamical system
\begin{equation}\label{dyn-syst}\left\{
\begin{array}{ll}
\ddot x(t) + \gamma(t)\dot x(t) + \lambda(t)B(x(t))=0\\
x(0)=u_0, \dot x(0)=v_0.
\end{array}\right.\end{equation}

As in \cite{att-sv2011, abbas-att-sv}, we consider the following definition of an absolutely continuous function.

\begin{definition}\label{abs-cont} \rm (see, for instance, \cite{att-sv2011, abbas-att-sv}) A function $x:[0,b]\rightarrow {\cal H}$ (where $b>0$) is said to be absolutely continuous if one of the 
following equivalent properties holds:

(i)  there exists an integrable function $y:[0,b]\rightarrow {\cal H}$ such that $$x(t)=x(0)+\int_0^t y(s)ds \ \ \forall t\in[0,b];$$

(ii) $x$ is continuous and its distributional derivative $\dot x$ is Lebesgue integrable on $[0,b]$; 

(iii) for every $\varepsilon > 0$, there exists $\eta >0$ such that for any finite family of intervals $I_k=(a_k,b_k)$ we have the implication
$$\left(I_k\cap I_j=\emptyset \mbox{ and }\sum_k|b_k-a_k| < \eta\right)\Longrightarrow \sum_k\|x(b_k)-x(a_k)\| < \varepsilon.$$
\end{definition}

\begin{remark}\label{rem-abs-cont}\rm (a) It follows from the definition that an absolutely continuous function is differentiable almost 
everywhere, its derivative coincides with its distributional derivative almost everywhere and one can recover the function from its derivative $\dot x=y$ by the integration formula (i). 

(b) If $x:[0,b]\rightarrow {\cal H}$ (where $b>0$) is absolutely continuous, then the function $z=B\circ x$ is absolutely continuous, too. This can be easily seen by using the characterization of absolute continuity in
Definition \ref{abs-cont}(iii). Moreover, $z$ is almost everywhere differentiable and the inequality $\|\dot z (\cdot)\|\leq L\|\dot x(\cdot)\|$ holds almost everywhere.   
\end{remark}

\begin{definition}\label{str-sol}\rm We say that $x:[0,+\infty)\rightarrow {\cal H}$ is a strong global solution of \eqref{dyn-syst} if the 
following properties are satisfied: 

(i) $x,\dot x:[0,+\infty)\rightarrow {\cal H}$ are locally absolutely continuous, in other words, absolutely continuous on each interval $[0,b]$ for $0<b<+\infty$; 

(ii) $\ddot x(t) + \gamma(t)\dot x(t) + \lambda(t)B(x(t))=0$ for almost every $t\in[0,+\infty)$;

(iii) $x(0)=u_0$ and $\dot x(0)=v_0$.
\end{definition}

For proving the existence and uniqueness of strong global solutions of \eqref{dyn-syst} we use the Cauchy-Lipschitz-Picard Theorem for absolutely continues trajectories. The key observation here is that one can rewrite \eqref{dyn-syst} 
as a particular first order dynamical system in a suitably chosen product space (see also \cite{alv-att-bolte-red}). 

\begin{theorem}\label{existence-th} Let $B:{\cal H}\rightarrow {\cal H}$ be an $L$-Lipschitz continuous operator and 
$\lambda,\gamma:[0,+\infty)\rightarrow [0,+\infty)$ be Lebesgue measurable functions such that $\lambda,\gamma \in L^1_{\loc}([0,+\infty))$ 
(that is $\lambda,\gamma \in L^1([0,b])$ for every $0<b<+\infty$). Then for each $u_0,v_0\in {\cal H}$ there exists a unique 
strong global solution of the dynamical system \eqref{dyn-syst}.  
\end{theorem}

\begin{proof} 
The system \eqref{dyn-syst} can be equivalently written as a first order dynamical system in the phase space ${\cal H}\times {\cal H}$
\begin{equation}\label{dyn-syst-eq}\left\{
\begin{array}{ll}
\dot Y(t)=F(t,Y(t))\\
Y(0)=(u_0,v_0),
\end{array}\right.\end{equation}
with 
$$Y:[0,+\infty)\rightarrow {\cal H}\times {\cal H}, \ Y(t)=(x(t),\dot x(t))$$ 
and 
$$F:[0,+\infty)\times{\cal H}\times {\cal H}\rightarrow {\cal H}\times {\cal H}, \ F(t,u,v)=(v,-\gamma(t) v-\lambda(t)Bu).$$ 

We endow ${\cal H}\times {\cal H}$ with scalar product $\langle (u,v),(\ol u,\ol v)\rangle_{{\cal H}\times {\cal H}}= \langle u,\ol u\rangle+\langle v,\ol v\rangle$ and corresponding norm 
$\|(u,v)\|_{{\cal H}\times {\cal H}}=\sqrt{\|u\|^2+\|v\|^2}$. 

(a) For arbitrary $u,\ol u,v,\ol v\in {\cal H}$, by using the Lipschitz continuity of the involved operators, we obtain 
\begin{align*}
\|F(t,u,v)-F(t,\ol u,\ol v)\|_{{\cal H}\times {\cal H}} & = \sqrt{\|v-\ol v\|^2+\|\gamma(t)( \ol v- v) + \lambda(t)(B\ol u-Bu)\|^2}\\
& \leq \sqrt{(1+2\gamma^2(t))\|v-\ol v\|^2 + 2L^2\lambda^2(t)\|u-\ol u\|^2}\\
& \leq \sqrt{1+2\gamma^2(t)+2L^2\lambda^2(t)}\|(u,\ol u)-(v,\ol v)\|_{{\cal H}\times {\cal H}}\\
& \leq (1+\sqrt{2}\gamma(t)+\sqrt{2} L\lambda(t))\|(u,\ol u)-(v,\ol v)\|_{{\cal H}\times {\cal H}} \ \forall t \geq 0.
\end{align*}
As $\lambda,\gamma \in L^1_{\loc}([0,+\infty))$, the Lipschitz constant of $F(t,\cdot,\cdot)$ is locally integrable.  

(b) Next we show that 
\begin{equation}\label{lip-th-2}\forall u,v\in{\cal H}, \ \forall b>0, \ \ F(\cdot,u,v)\in L^1([0,b],{\cal H}\times {\cal H}).\end{equation}
For arbitrary $u,v\in {\cal H}$ and $b>0$ it holds 
\begin{align*}
\int_0^b\|F(t,u,v)\|_{{\cal H}\times {\cal H}}dt = &  \int_0^b\sqrt{\|v\|^2+\|\gamma(t) v+\lambda(t)Bu\|^2}dt\\
\leq & \int_0^b\sqrt{(1+2\gamma^2(t))\|v\|^2+2\lambda^2(t)\|Bu\|^2}dt\\
\leq & \int_0^b\left((1+\sqrt{2}\gamma(t))\|v\|+\sqrt{2}\lambda(t)\|Bu\|\right)dt
\end{align*}
and from here, by using the assumptions made on $\lambda,\gamma$, \eqref{lip-th-2} follows.

In the light of the statements (a) and (b), the existence and uniqueness of a strong global solution for \eqref{dyn-syst-eq} follow from the Cauchy-Lipschitz-Picard Theorem for 
first order dynamical systems (see, for example, \cite[Proposition 6.2.1]{haraux}). The conclusion is a consequence of the equivalence of \eqref{dyn-syst} and \eqref{dyn-syst-eq}.
\end{proof}

\section{Convergence of the trajectories}\label{sec3}

In this section we address the convergence properties of the trajectories generated by the dynamical system \eqref{dyn-syst} by assuming that 
$B : {\cal H} \rightarrow {\cal H}$ is {\it $\beta$-cocoercive} for $\beta > 0$, 
that is $\beta \|Bx - By\|^2 \leq \langle x-y, Bx-By \rangle$ for all $x,y \in {\cal H}$. This implies that $B$ is $\frac{1}{\beta}$-Lipschitz continuous. If $B = \nabla g$, where $g : {\cal H} \rightarrow \R$ is a convex
and differentiable function such that $\nabla g$ is $\frac{1}{\beta}$-Lipschitz continuous, then the reverse implication holds, too. Indeed, according to the Baillon-Haddad Theorem, 
$\nabla g$ is a $\beta$-cocoercive operator (see \cite[Corollary 18.16]{bauschke-book}).

The following results, which can be interpreted as continuous versions of the quasi-Fej\'er monotonicity for sequences, plays an important role in the forthcoming investigations. For their proofs we refer the reader  to 
\cite[Lemma 5.1]{abbas-att-sv} and \cite[Lemma 5.2]{abbas-att-sv}, respectively.

\begin{lemma}\label{fejer-cont1} Suppose that $F:[0,+\infty)\rightarrow\R$ is locally absolutely continuous and bounded below and that
there exists $G\in L^1([0,+\infty))$ such that for almost every $t \in [0,+\infty)$ $$\frac{d}{dt}F(t)\leq G(t).$$ 
Then there exists $\lim_{t\rightarrow \infty} F(t)\in\R$. 
\end{lemma}

\begin{lemma}\label{fejer-cont2}  If $1 \leq p < \infty$, $1 \leq r \leq \infty$, $F:[0,+\infty)\rightarrow[0,+\infty)$ is 
locally absolutely continuous, $F\in L^p([0,+\infty))$, $G:[0,+\infty)\rightarrow\R$, $G\in  L^r([0,+\infty))$ and 
for almost every $t \in [0,+\infty)$ $$\frac{d}{dt}F(t)\leq G(t),$$ then $\lim_{t\rightarrow +\infty} F(t)=0$. 
\end{lemma}

The next result which we recall here is the continuous version of the Opial Lemma (see, for example, \cite[Lemma 5.3]{abbas-att-sv}, \cite[Lemma 1.10]{abbas-att-arx14}). 

\begin{lemma}\label{opial} Let $S \subseteq {\cal H}$ be a nonempty set and $x:[0,+\infty)\rightarrow{\cal H}$ a given map. Assume that 

(i) for every $x^*\in S$, $\lim_{t\rightarrow+\infty}\|x(t)-x^*\|$ exists; 

(ii) every weak sequential cluster point of the map $x$ belongs to $S$. 

\noindent Then there exists $x_{\infty}\in S$ such that  $x(t)$ converges weakly to $x_{\infty}$ as $t \rightarrow +\infty$. 
\end{lemma}

In order to prove the convergence of the trajectories of \eqref{dyn-syst}, we make the following assumptions on the relaxation function $\lambda$ and 
the damping parameter $\gamma$, respectively: 
\begin{enumerate}
\item[{\rm (A1)}] $\lambda, \gamma :[0,+\infty)\rightarrow (0,+\infty)$ are locally absolutely continuous and there exists $\theta >0$ such that for almost every $t\in [0, +\infty)$ we have  
\begin{equation}\label{h3-g}\dot\gamma(t)\leq 0\leq\dot\lambda(t) \mbox{ and } \frac{\gamma^2(t)}{\lambda(t)}\geq\frac{1+\theta}{\beta}.\end{equation}
\end{enumerate}

Due to Definition \ref{abs-cont} and Remark \ref{rem-abs-cont}(a), $\dot\lambda(t),\dot\gamma(t)$ exists for almost every $t\geq 0$ and 
$\dot\lambda,\dot\gamma$ are Lebesgue integrable on each interval $[0,b]$ for $0<b<+\infty$. This combined with 
$\dot\gamma(t)\leq 0\leq\dot\lambda(t)$ and the fact that $\lambda,\gamma$ take only positive values yield 
the existence of a positive lower bound $\ul\lambda$ for $\lambda$ and of a positive upper bound $\ol\gamma$
for $\gamma$. Furthermore, the second assumption in \eqref{h3-g} provides also a positive upper bound $\ol\lambda$ for $\lambda$ and a positive 
lower bound $\ul\gamma$ for $\gamma$. Notice that the couple of functions 
$$\lambda(t)=\frac{1}{ae^{-\rho t}+b} \mbox{ and } \gamma(t)=a'e^{-\rho't}+b',$$
where $a,a',\rho,\rho'\geq 0$ and $b,b'>0$ fulfill the inequality $b'^2b>\frac{1}{\beta}$, verify  the conditions in assumption (A1).  

We would also like to point out that under the conditions considered in (A1) the global version of the Picard-Lindel\"{o}f Theorem allows us 
to conclude that, for $u_0,v_0\in {\cal H}$, there exists a unique trajectory $x:[0,+\infty)\rightarrow{\cal H}$ which is a 
$C^2$ function and which satisfies the relation (ii) in Definition 2 for every $t\in [0,+\infty)$. The considerations we make in the following take into account this fact. 

We state now the convergence result.

\begin{theorem}\label{conv-th}  Let $B: {\cal H}\rightarrow{\cal H}$ be a $\beta$-cocoercive operator for $\beta>0$ such that 
$\zer B:=\{u\in {\cal H}:Bu=0\}\neq\emptyset$, $\lambda,\gamma:[0,+\infty)\rightarrow(0,+\infty)$ be functions fulfilling {\rm (A1)} and 
$u_0,v_0\in {\cal H}$. Let $x:[0,+\infty)\rightarrow {\cal H}$ be the unique strong global solution of \eqref{dyn-syst}. 
Then the following statements are true:

(i) the trajectory $x$ is bounded and $\dot x,\ddot x,Bx\in L^2([0,+\infty); {\cal H})$;  

(ii) $\lim_{t\rightarrow+\infty}\dot x(t)=\lim_{t\rightarrow+\infty}\ddot x(t)=\lim_{t\rightarrow+\infty} B(x(t)=0$; 

(iii) $x(t)$ converges weakly to an element in $\zer B$ as $t\rightarrow+\infty$.
\end{theorem}

\begin{proof} (i) Take an arbitrary $x^*\in\zer B$ and consider for every $t \in [0, +\infty)$ the function $h(t)=\frac{1}{2}\|x(t)-x^*\|^2$. We have $\dot h(t)=\langle x(t)-x^*,\dot x(t)\rangle$ and 
$\ddot h(t)=\|\dot x(t)\|^2+ \<x(t)-x^* , \ddot x(t)\>$ for every $t \in [0,+\infty)$. Taking into account \eqref{dyn-syst}, we get for every $t \in [0,+\infty)$
\begin{equation}\label{eq-h}\ddot h(t) + \gamma(t)\dot h(t) + \lambda(t)\<x(t)-x^* , B(x(t))\> = \|\dot x(t)\|^2.\end{equation}

The cocoercivity of $B$ and the fact that $Bx^*=0$ yields for every $t \in [0, +\infty)$
$$\ddot h(t) + \gamma(t)\dot h(t) + \beta\lambda(t)\|B(x(t))\|^2  \leq \|\dot x(t)\|^2.$$
Taking again into account \eqref{dyn-syst} one obtains for every $t \in [0, +\infty)$
$$\ddot h(t) + \gamma(t)\dot h(t) + \frac{\beta}{\lambda(t)}\|\ddot x(t)+\gamma(t)\dot x(t)\|^2  \leq \|\dot x(t)\|^2$$
or, equivalently,
\begin{equation*}\ddot h(t) + \gamma(t)\dot h(t)  + 
\frac{\beta\gamma(t)}{\lambda(t)}\frac{d}{dt}\big(\|\dot x(t)\|^2\big)  + \left(\frac{\beta\gamma^2(t)}{\lambda(t)}-1\right)||\dot x(t)||^2 +
\frac{\beta}{\lambda(t)}||\ddot x(t)||^2  \leq 0.\end{equation*}
Combining this inequality with 
\begin{equation}\label{d-g-h-x}\frac{\gamma(t)}{\lambda(t)}\frac{d}{dt}\big(\|\dot x(t)\|^2\big)  =  \ 
\frac{d}{dt}\left(\frac{\gamma(t)}{\lambda(t)}\|\dot x(t)\|^2\right)-
\frac{\dot \gamma(t)\lambda(t)-\gamma(t)\dot\lambda(t)}{\lambda^2(t)}\|\dot x(t)\|^2\\
\end{equation}
and \begin{equation}\label{d-g-h}\gamma(t)\dot h(t)=\frac{d}{dt}(\gamma h)(t)-\dot\gamma(t)h(t)\geq \frac{d}{dt}(\gamma h)(t),\end{equation}
it yields for every $t \in [0,+\infty)$
\begin{align*}
& \ddot h(t)  + \frac{d}{dt}(\gamma h)(t) +\\
& \beta\frac{d}{dt}\left(\frac{\gamma(t)}{\lambda(t)}\|\dot x(t)\|^2\right)  + \left(\frac{\beta\gamma^2(t)}{\lambda(t)}+
\beta\frac{-\dot \gamma(t)\lambda(t)+\gamma(t)\dot\lambda(t)}{\lambda^2(t)}-1\right)||\dot x(t)||^2 +
\frac{\beta}{\lambda(t)}||\ddot x(t)||^2  \leq 0.\end{align*}
Now, assumption (A1) delivers for almost every $t \in [0, +\infty)$ the inequality 
\begin{equation}\label{ineq-h-renorm-fin}\ddot h(t)  + \frac{d}{dt}(\gamma h)(t) +
\beta\frac{d}{dt}\left(\frac{\gamma(t)}{\lambda(t)}\|\dot x(t)\|^2\right)  + \theta||\dot x(t)||^2 +
\beta\ol\lambda^{-1}\|\ddot x(t)||^2  \leq 0.\end{equation} 
This implies that the function $t\mapsto \dot h(t)+\gamma(t) h(t)+\beta\frac{\gamma(t)}{\lambda(t)}\|\dot x(t)\|^2$, which is locally absolutely
continuous, is monotonically decreasing. Hence there exists a real number $M$ such that for every $t \in [0, +\infty)$ 
\begin{equation}\label{ineq-h-bound} \dot h(t)+\gamma(t) h(t)+\beta\frac{\gamma(t)}{\lambda(t)}\|\dot x(t)\|^2\leq M, 
\end{equation}
which yields that for every $t \in [0, +\infty)$
$$\dot h(t)+\ul\gamma h(t)\leq M.$$
By multiplying this inequality with $\exp(\ul\gamma t)$ and then integrating from $0$ to $T$, where $T > 0$, one easily obtains
$$h(T)\leq h(0)\exp(-\ul\gamma T)+\frac{M}{\ul\gamma}(1-\exp(-\ul\gamma T)),$$
thus 
\begin{equation}\label{h-bound} h \mbox{ is bounded}\end{equation}
and, consequently, 
\begin{equation}\label{x-bound} \mbox{the trajectory }x \mbox{ is bounded}.\end{equation}

On the other hand, from \eqref{ineq-h-bound}, it follows that for every $t \in [0, +\infty)$
$$ \dot h(t)+\beta\ul\gamma\ol\lambda^{-1}\|\dot x(t)\|^2\leq M,$$
hence 
$$ \<x(t)-x^*,\dot x(t)\>+\beta\ul\gamma\ol\lambda^{-1}\|\dot x(t)\|^2\leq M.$$
This inequality in combination with \eqref{x-bound} yields 
\begin{equation}\label{dotx-bound} \dot x \mbox{ is bounded},\end{equation}
which further implies that
\begin{equation}\label{doth-bound} \dot h \mbox{ is bounded}.\end{equation}

Integrating the inequality \eqref{ineq-h-renorm-fin} we obtain that there exists a real number $N \in \R$ such that for every $t \in [0, +\infty)$
$$\dot h(t)+\gamma(t) h(t)+\beta\frac{\gamma(t)}{\lambda(t)}\|\dot x(t)\|^2+\theta\int_0^t||\dot x(s)||^2ds
+\beta\ol\lambda^{-1}\int_0^t||\ddot x(s)||^2ds\leq N.$$
From here, via \eqref{doth-bound}, we conclude that $\dot x(\cdot),\ddot x(\cdot)\in L^2([0,+\infty); {\cal H})$. 
Finally, from \eqref{dyn-syst} and (A1) we deduce $Bx \in L^2([0,+\infty); {\cal H})$ and the proof of (i) is complete. 

(ii) For every $t \in [0, +\infty)$ it holds
$$\frac{d}{dt}\left(\frac{1}{2}\|\dot x(t)\|^2\right)=\<\dot x(t),\ddot x(t)\>\leq \frac{1}{2}\|\dot x(t)\|^2+\frac{1}{2}\|\ddot x(t)\|^2$$
and Lemma \ref{fejer-cont2} together with (i) lead to $\lim_{t\rightarrow+\infty}\dot x(t)=0$.

Further, by taking into consideration Remark \ref{rem-abs-cont}(b), for every $t \in [0, +\infty)$ we have 
$$\frac{d}{dt}\left(\frac{1}{2}\|B(x(t))\|^2\right)=\<B(x(t)),\frac{d}{dt}(Bx(t))\>\leq \frac{1}{2}\|B(x(t))\|^2+\frac{1}{2\beta^2}\|\dot x(t)\|^2.$$
By using again Lemma \ref{fejer-cont2} and (i) we get  
$\lim_{t\rightarrow+\infty}B(x(t))=0$, while the fact that $\lim_{t\rightarrow+\infty}\ddot x(t)=0$ follows from 
\eqref{dyn-syst} and (A2). 

(iii) We are going to prove that both assumptions in Opial Lemma are fulfilled. The first one concerns the existence 
of $\lim_{t\rightarrow +\infty }\|x(t)-x^*\|$. As seen in the proof of part (i), the function 
$t\mapsto \dot h(t)+\gamma(t) h(t)+\beta\frac{\gamma(t)}{\lambda(t)}\|\dot x(t)\|^2$ is 
monotonically decreasing, thus from (i), (ii) and (A1) we deduce that 
$\lim_{t\rightarrow+\infty} \gamma(t) h(t)$ exists and it is a real number. 
By taking also into account that $\exists \lim_{t\rightarrow+\infty}\gamma(t)\in(0,\infty)$, we obtain 
the existence of $\lim_{t\rightarrow +\infty }\|x(t)-x^*\|$. 

We come now to the second assumption of the Opial Lemma. Let $\ol x$ be a weak sequential cluster point of $x$, that is, there exists 
a sequence $t_n\rightarrow+\infty$ (as $n\rightarrow+\infty$) such that $(x(t_n))_{n\in\N}$ converges weakly to $\ol x$. Since $B$ is a maximally monotone operator (see for instance \cite[Example 20.28]{bauschke-book}), its graph is sequentially closed with respect to 
the weak-strong topology of the product space ${\cal H}\times {\cal H}$. By using also that $\lim_{n\rightarrow+\infty}B(x({t_n}))=0$, we conclude 
that $B\ol x=0$, hence $\ol x\in\zer B$ and the proof is complete. \end{proof}

A standard choice of a cocoercive operator defined on a real Hilbert spaces is $B=\id-T$, where $T: {\cal H}\rightarrow {\cal H}$ is a {\it nonexpansive operator}, that is, a $1$-Lipschitz
continuous operator. As it easily follows from the nonexpansiveness of $T$, $B$ is in this case $1/2$-cocoercive. For this particular operator $B$ the dynamical system \eqref{dyn-syst} becomes 
\begin{equation}\label{dyn-syst-nonexp}\left\{
\begin{array}{ll}
\ddot x(t) + \gamma(t) \dot x(t) + \lambda(t)\big(x(t)-T(x(t))\big)=0\\
x(0)=u_0, \dot x(0)=v_0,
\end{array}\right.\end{equation}
while assumption (A1) reads
\begin{enumerate}
\item[{\rm (A2)}] $\lambda, \gamma :[0,+\infty)\rightarrow (0,+\infty)$ are locally absolutely continuous and there exists $\theta >0$ such that for almost every $t\in [0, +\infty)$ we have  
\begin{equation}\label{h3'-g}\dot\gamma(t)\leq 0\leq\dot\lambda(t) \mbox{ and } \frac{\gamma^2(t)}{\lambda(t)}\geq 2 (1+\theta).\end{equation}
\end{enumerate}

Theorem \ref{conv-th} gives rise to the following result.

\begin{corollary}\label{conv-th-nonexp} Let $T: {\cal H}\rightarrow{\cal H}$ be a nonexpansive operator such that 
$\fix T:=\{u\in {\cal H}:Tu=u\}\neq\emptyset$, $\lambda,\gamma:[0,+\infty)\rightarrow(0,+\infty)$ be functions 
fulfilling {\rm (A2)} and $u_0,v_0\in {\cal H}$. 
Let $x:[0,+\infty)\rightarrow {\cal H}$ be the unique strong global solution of \eqref{dyn-syst-nonexp}. Then the following statements are true: 

(i) the trajectory $x$ is bounded and $\dot x,\ddot x,(\id -T)x\in L^2([0,+\infty); {\cal H})$;  

(ii) $\lim_{t\rightarrow+\infty}\dot x(t)=\lim_{t\rightarrow+\infty}\ddot x(t)=\lim_{t\rightarrow+\infty}(\id -T)(x(t))=0$; 

(iii) $x(t)$ converges weakly to a point in $\fix T$ as $t\rightarrow+\infty$.
\end{corollary}

\begin{remark}\label{alv-att-nonexp} In the particular case when $\gamma(t)=\gamma > 0$ for all $t\geq 0$ and 
$\lambda(t)=1$ for all $t \in [0, +\infty)$ the dynamical system \eqref{dyn-syst-nonexp} becomes 
\begin{equation}\label{dyn-syst-nonexp-alv-att}\left\{
\begin{array}{ll}
\ddot x(t) + \gamma \dot x(t) + x(t)-T(x(t))=0\\
x(0)=u_0, \dot x(0)=v_0.
\end{array}\right.\end{equation}
The convergence of the trajectories generated by \eqref{dyn-syst-nonexp-alv-att} has been studied in \cite[Theorem 3.2]{att-alv} 
under the condition $\gamma^2>2$. In this case (A2) is obviously fulfilled for an arbitrary 
$0<\theta\leq(\gamma^2-2)/2$. However, different to \cite{att-alv}, we allow in Corollary \ref{conv-th-nonexp} 
nonconstant damping and relaxation functions depending on time. We would also like to notice that in \cite{alvarez2000} an anisotropic damping has 
been considered in the context of approaching the minimization of a smooth convex function via second order dynamical systems. 
\end{remark}

We close the section by addressing an immediate consequence of the above corollary applied to second order dynamical systems governed by averaged operators. 
The operator $R:{\cal H}\rightarrow{\cal H}$ is said to be {\it $\alpha$-averaged} for $\alpha \in (0,1)$, if there exists a nonexpansive operator 
$T:{\cal H}\rightarrow{\cal H}$ such that $R=(1-\alpha)\id+\alpha T$. For $\alpha=\frac{1}{2}$ we obtain as an important representative of this class the
{\it firmly nonexpansive} operators. For properties and insights concerning these families of operators we refer to the monograph \cite{bauschke-book}. 

We consider the dynamical system
\begin{equation}\label{dyn-syst-av}\left\{
\begin{array}{ll}
\ddot x(t) + \gamma(t) \dot x(t) + \lambda(t)\big(x(t)-R(x(t))\big)=0\\
x(0)=u_0, \dot x(0)=v_0
\end{array}\right.\end{equation}
and formulate the assumption
\begin{enumerate}
\item[{\rm (A3)}] $\lambda, \gamma :[0,+\infty)\rightarrow (0,+\infty)$ are locally absolutely continuous and there exists $\theta >0$ such that for almost every $t\in [0, +\infty)$ we have  
\begin{equation}\label{h3''-g}\dot\gamma(t)\leq 0\leq\dot\lambda(t) \mbox{ and } \frac{\gamma^2(t)}{\lambda(t)}\geq 2\alpha (1+\theta).\end{equation}
\end{enumerate}

\begin{corollary}\label{conv-th-av} Let $R: {\cal H}\rightarrow{\cal H}$ be an $\alpha$-averaged operator for $\alpha\in(0,1)$ such 
that $\fix R\neq\emptyset$, $\lambda,\gamma:[0,+\infty)\rightarrow(0,+\infty)$ be functions fulfilling {\rm (A3)} 
and $u_0,v_0\in {\cal H}$. Let $x:[0,+\infty)\rightarrow {\cal H}$ be the unique strong global solution of \eqref{dyn-syst-av}. 
Then the following statements are true: 

(i) the trajectory $x$ is bounded and $\dot x,\ddot x,(\id -R)x\in L^2([0,+\infty); {\cal H})$;  

(ii) $\lim_{t\rightarrow+\infty}\dot x(t)=\lim_{t\rightarrow+\infty}\ddot x(t)=\lim_{t\rightarrow+\infty}(\id -R)(x(t))=0$; 

(iii) $x(t)$ converges weakly to a point in $\fix R$ as $t\rightarrow+\infty$.
\end{corollary}

\begin{proof} Since $R$ is $\alpha$-averaged, there exists a nonexpansive operator 
$T:{\cal H}\rightarrow{\cal H}$ such that $R=(1-\alpha)\id+\alpha T$. The conclusion is a direct consequence of Corollary \ref{conv-th-nonexp}, 
by taking into account that \eqref{dyn-syst-av} is equivalent to $$\left\{
\begin{array}{ll}
\ddot x(t) + \gamma(t) \dot x(t) + \alpha\lambda(t)\big(x(t)-T(x(t))\big)=0\\
x(0)=u_0, \dot x(0)=v_0,
\end{array}\right.$$ and $\fix R=\fix T$.  
\end{proof}

\section{Forward-backward second order dynamical systems}\label{sec4}

In this section we approach the monotone inclusion problem
\begin{equation*}
\mbox{find} \ 0 \in A(x) + B(x),
\end{equation*}
where $A:{\cal H}\rightrightarrows {\cal H}$ is a maximally monotone operator and $B:{\cal H}\rightarrow {\cal H}$ is a $\beta$-cocoercive operator for $\beta >0$  
via a second order forward-backward dynamical system with relaxation and damping functions depending on time.

For readers convenience we recall at the beginning some standard notions and results in monotone operator theory (see also \cite{bo-van, bauschke-book, simons}). 
For an arbitrary set-valued operator $A:{\cal H}\rightrightarrows {\cal H}$ we denote by 
$\gr A=\{(x,u)\in {\cal H}\times {\cal H}:u\in Ax\}$ its graph. We use also the notation $\zer A=\{x\in{\cal{H}}:0\in Ax\}$ for the set of zeros of $A$. 
We say that $A$ is monotone, if $\langle x-y,u-v\rangle\geq 0$ for all $(x,u),(y,v)\in\gr A$. A monotone operator $A$ is said to be maximally monotone, if there exists no proper monotone extension of the graph of $A$ on 
${\cal H}\times {\cal H}$. The resolvent of $A$, $J_A:{\cal H} \rightrightarrows {\cal H}$, is defined by $J_A=(\id +A)^{-1}$. If $A$ is maximally monotone, then $J_A:{\cal H} \rightarrow {\cal H}$ is single-valued and 
maximally monotone (see \cite[Proposition 23.7 and Corollary 23.10]{bauschke-book}). For an arbitrary $\gamma>0$ we have (see \cite[Proposition 23.2]{bauschke-book})
\begin{equation}p\in J_{\gamma A}x \ \mbox{if and only if} \ (p,\gamma^{-1}(x-p))\in\gr A.\end{equation}

The operator $A$ is said to be uniformly monotone if there exists an increasing function
$\phi_A : [0,+\infty) \rightarrow [0,+\infty]$ that vanishes only at $0$ and fulfills
$\langle x-y,u-v \rangle \geq \phi_A \left( \| x-y \|\right)$ for all $(x,u), (y,v) \in \gr A$. A popular
class of operators having this property is the one of strongly monotone operators. We say that $A$ is $\gamma$-strongly monotone for $\gamma > 0$, 
if $\langle x-y,u-v\rangle\geq \gamma\|x-y\|^2$ for all $(x,u),(y,v)\in\gr A$. 

For $\eta > 0$ we consider the dynamical system
\begin{equation}\label{dyn-syst-fb}\left\{
\begin{array}{ll}
\ddot x(t) + \gamma(t)\dot x(t) + \lambda(t)\left[x(t)-J_{\eta A}\Big(x(t)-\eta B(x(t))\Big)\right]=0\\
x(0)=u_0, \dot x(0)=v_0.
\end{array}\right.\end{equation}
Further, we consider the following assumption, where $\delta:=\frac{4\beta-\eta}{2\beta}$: 
\begin{enumerate}
\item[{\rm (A4)}] $\lambda, \gamma :[0,+\infty)\rightarrow (0,+\infty)$ are locally absolutely continuous and there exists $\theta >0$ such that for almost every $t\in [0, +\infty)$ we have  
\begin{equation}\label{h3'''-g}\dot\gamma(t)\leq 0\leq\dot\lambda(t) \mbox{ and } \frac{\gamma^2(t)}{\lambda(t)}\geq \frac{2 (1+\theta)}{\delta}.\end{equation}
\end{enumerate}

\begin{theorem}\label{fb-dyn} Let $A:{\cal H}\rightrightarrows {\cal H}$ be a maximally monotone operator and 
$B:{\cal H}\rightarrow {\cal H}$ be a $\beta$-cocoercive operator for $\beta > 0$ 
such that $\zer(A+B)\neq\emptyset$. Let $\eta\in(0,2\beta)$ and set $\delta:=\frac{4\beta-\eta}{2\beta}$. Let 
$\lambda,\gamma:[0,+\infty)\rightarrow(0,+\infty)$ be functions fulfilling {\rm (A4)}, $u_0,v_0\in {\cal H}$ 
and $x:[0,+\infty)\rightarrow {\cal H}$ be the unique strong global solution of \eqref{dyn-syst-fb}. 
Then the following statements are true: 

(i) the trajectory $x$ is bounded and $\dot x,\ddot x,\big(\id -J_{\eta A}\circ(\id -\eta B)\big)x\in L^2([0,+\infty); {\cal H})$;  

(ii) $\lim_{t\rightarrow+\infty}\dot x(t)=\lim_{t\rightarrow+\infty}\ddot x(t)=
\lim_{t\rightarrow+\infty}\big(\id -J_{\eta A}\circ(\id -\eta B)\big)(x(t))=0$; 

(iii) $x(t)$ converges weakly to a point in $\zer(A+B)$ as $t\rightarrow+\infty$;

(iv) if $x^*\in\zer(A+B)$, then $B(x(\cdot))-Bx^*\in L^2([0,+\infty); {\cal H})$, 
$\lim_{t\rightarrow+\infty}B(x(t))=Bx^*$ and $B$ is constant on $\zer(A+B)$; 

(v) if $A$ or $B$ is uniformly monotone, then $x(t)$ converges strongly to the unique point in $\zer(A+B)$ as $t\rightarrow+\infty$.
\end{theorem}

\begin{proof} (i)-(iii) It is immediate that the dynamical system \eqref{dyn-syst-fb} can be written in the form 
\begin{equation}\label{dyn-syst-av'}\left\{
\begin{array}{ll}
\ddot x(t) + \gamma(t) \dot x(t) + \lambda(t)\big(x(t)-R(x(t))\big)=0\\
x(0)=u_0, \dot x(0)=v_0,
\end{array}\right.\end{equation}
where $R=J_{\eta A}\circ(\id -\eta B).$ According to \cite[Corollary 23.8 and Remark 4.24(iii)]{bauschke-book}, $J_{\eta A}$ is $1/2$-cocoercive. 
Moreover, by \cite[Proposition 4.33]{bauschke-book}, $\id -\eta B$ is $\eta/(2\beta)$-averaged. Combining this with 
\cite[Theorem 3(b)]{og-yam}, we derive that $R$ is $1/\delta$-averaged. The statements (i)-(iii) follow now from 
Corollary \ref{conv-th-av} by noticing that $\fix R=\zer(A+B)$ (see \cite[Proposition 25.1(iv)]{bauschke-book}). 

(iv) The fact that $B$ is constant on $\zer(A+B)$ follows from the cocoercivity of $B$ and the monotonicity of $A$. A proof of this statement when $A$ is 
the subdifferential of a proper, convex and lower semicontinuous function is given for instance in \cite[Lemma 1.7]{abbas-att-arx14}. 

Let be an arbitrary $x^* \in \zer(A+B)$. From the definition of the resolvent we have for every
$t \in [0, +\infty)$
\begin{equation}\label{conseq-def-res-g} -B(x(t))-\frac{1}{\eta\lambda(t)}\ddot x(t)-\frac{\gamma(t)}{\eta\lambda(t)}\dot x(t)\in 
A\left(\frac{1}{\lambda(t)}\ddot x(t)+\frac{\gamma(t)}{\lambda(t)}\dot x(t)+x(t)\right),\end{equation}
which combined with $-Bx^* \in Ax^*$ and the monotonicity of $A$ leads to
\begin{equation}\label{ineq-fb-mon-g} 0\leq\< \frac{1}{\lambda(t)}\ddot x(t)+\frac{\gamma(t)}{\lambda(t)}\dot x(t)+x(t)-x^*, -B(x(t))+ Bx^*-\frac{1}{\eta\lambda(t)}\ddot x(t)-\frac{\gamma(t)}{\eta\lambda(t)}\dot x(t)\>. 
\end{equation}
The cocoercivity of $B$ yields for every $t \in [0, +\infty)$
\begin{align*}
\beta\|B(x(t))-Bx^*\|^2 \leq & \<\frac{1}{\lambda(t)}\ddot x(t)+\frac{\gamma(t)}{\lambda(t)}\dot x(t) , -B(x(t))+Bx^*\>\!\!-\frac{1}{\eta\lambda^2(t)}\|\ddot x(t)+\gamma(t)\dot x(t)\|^2\\
&+\<x(t)-x^* , -\frac{1}{\eta\lambda(t)}\ddot x(t)-\frac{\gamma(t)}{\eta\lambda(t)}\dot x(t)\>\\
\leq & \frac{1}{2\beta}\left\|\frac{1}{\lambda(t)}\ddot x(t)+\frac{\gamma(t)}{\lambda(t)}\dot x(t)\right\|^2+\frac{\beta}{2}\|B(x(t))-Bx^*\|^2\\
& -\frac{1}{\eta\lambda^2(t)}\|\ddot x(t)+\gamma(t)\dot x(t)\|^2 +\<x(t)-x^* , -\frac{1}{\eta\lambda(t)}\ddot x(t)-\frac{\gamma(t)}{\eta\lambda(t)}\dot x(t)\>\\
= & \frac{\eta - 2\beta}{2\eta\beta\lambda^2(t)} \|\ddot x(t)+\gamma(t)\dot x(t)\|^2 +\frac{\beta}{2}\|B(x(t))-Bx^*\|^2\\ 
&+\<x(t)-x^* , -\frac{1}{\eta\lambda(t)}\ddot x(t)-\frac{\gamma(t)}{\eta\lambda(t)}\dot x(t)\>\\
\leq & \frac{\beta}{2}\|B(x(t))-Bx^*\|^2 +\<x(t)-x^* , -\frac{1}{\eta\lambda(t)}\ddot x(t)-\frac{\gamma(t)}{\eta\lambda(t)}\dot x(t)\>.
\end{align*}
For evaluating the last term of the above inequality we use the function $h:[0,+\infty)\rightarrow \R$, $h(t)=\frac{1}{2}\|x(t)-x^*\|^2$, 
already used in the proof of Theorem \ref{conv-th}. 
From \begin{equation}\label{eq-h-p-g}\<x(t)-x^* , -\frac{1}{\eta\lambda(t)}\ddot x(t)-\frac{\gamma(t)}{\eta\lambda(t)}\dot x(t)\>=
-\frac{1}{\eta\lambda(t)}\left(\ddot h(t)+\gamma(t)\dot h(t)-\|\dot x(t)\|^2\right)\end{equation}
we obtain for every $t \in [0, +\infty)$
$$\frac{\beta\lambda(t)}{2}\|B(x(t))-Bx^*\|^2+\frac{1}{\eta}\left(\ddot h(t)+
\gamma(t)\dot h(t)\right) \leq \frac{1}{\eta}\|\dot x(t)\|^2.$$
Taking into account also the relation \eqref{d-g-h} and the bounds for $\lambda$, we get for every $t \in [0, +\infty)$
$$\frac{\beta\ul\lambda}{2}\|B(x(t))-Bx^*\|^2+\frac{1}{\eta}\left(\ddot h(t)+
\frac{d}{dt}(\gamma h)(t)\right) \leq \frac{1}{\eta}\|\dot x(t)\|^2.$$
After integration we obtain that for every $T \in [0,+\infty)$
\begin{align*}
\frac{\beta\ul\lambda}{2}\int_0^T\|B(x(t))-Bx^*\|^2dt+\frac{1}{\eta}\left(\dot h(T)-\dot h(0)+\gamma(T) h(T)-\gamma(0) h(0)\right) \leq \frac{1}{\eta} \int_0^T \|\dot x(t)\|^2 dt.
\end{align*}
Since $\dot x \in L^2([0,+\infty); {\cal H})$, $\gamma$ has a positive upper bound, $h(T)\geq 0, \gamma(T)\geq 0$ for every $T \in [0,+\infty)$ and 
$\lim_{T\rightarrow+\infty}\dot h(T)=0$, it follows that $B(x(\cdot))-Bx^*\in L^2([0,+\infty); {\cal H})$. 

Further, by taking into consideration Remark \ref{rem-abs-cont}(b), we have
\begin{align*}
\frac{d}{dt}\left(\frac{1}{2}\|B(x(t))-Bx^*\|^2\right)=  \<B(x(t))-Bx^*,\frac{d}{dt}(Bx(t))\> \leq \frac{1}{2}\|B(x(t))-Bx^*\|^2+\frac{1}{2\beta^2}\|\dot x(t)\|^2
\end{align*}
and from here, in the light of Lemma \ref{fejer-cont2},  it follows that $\lim_{t\rightarrow+\infty}B(x(t))=Bx^*$. 

(v) Let $x^*$ be the unique element of $\zer(A+B)$. For the beginning we suppose that $A$ is uniformly monotone with corresponding function $\phi_A:[0,+\infty) \rightarrow [0,+\infty]$, which is increasing and 
vanishes only at $0$. 

By similar arguments as in the proof of statement (iv), for every $t\in [0,+\infty)$ we have
\begin{align*}
&\phi_A\left(\left\|\frac{1}{\lambda(t)}\ddot x(t)+\frac{\gamma(t)}{\lambda(t)}\dot x(t)+x(t)-x^*\right\|\right)\leq\\
&\< \frac{1}{\lambda(t)}\ddot x(t)+\frac{\gamma(t)}{\lambda(t)}\dot x(t)+x(t)-x^*, -B(x(t))+ Bx^*-\frac{1}{\eta\lambda(t)}\ddot x(t)-\frac{\gamma(t)}{\eta\lambda(t)}\dot x(t)\>,
\end{align*}
which combined with the monotonicity of $B$ yields 
\begin{align*}
& \phi_A\left(\left\|\frac{1}{\lambda(t)}\ddot x(t)+\frac{\gamma(t)}{\lambda(t)}\dot x(t)+x(t)-x^*\right\|\right) \leq\\
& \<\frac{1}{\lambda(t)}\ddot x(t)+\frac{\gamma(t)}{\lambda(t)}\dot x(t) , -B(x(t))+Bx^*\>  -\frac{1}{\eta\lambda^2(t)}\|\ddot x(t)+\gamma(t)\dot x(t)\|^2+\\
& \<x(t)-x^* , -\frac{1}{\eta\lambda(t)}\ddot x(t)-\frac{\gamma(t)}{\eta\lambda(t)}\dot x(t)\> \leq \\
& \<\frac{1}{\lambda(t)}\ddot x(t)+\frac{\gamma(t)}{\lambda(t)}\dot x(t) , -B(x(t))+Bx^*\> +\<x(t)-x^* , -\frac{1}{\eta\lambda(t)}\ddot x(t)-\frac{\gamma(t)}{\eta\lambda(t)}\dot x(t)\>.
\end{align*}
As $\lambda$ and $\gamma$ are bounded by positive constants, by using (i)-(iv) it follows that the right-hand side of the last inequality converges to 0 as $t\rightarrow+\infty$. Hence  
$$\lim_{t\rightarrow+ \infty}\phi_A\left(\left\|\frac{1}{\lambda(t)}\ddot x(t)+\frac{\gamma(t)}{\lambda(t)}\dot x(t)+x(t)-x^*\right\|\right)=0$$ 
and the properties of the function $\phi_A$ allow to conclude that 
$\frac{1}{\lambda(t)}\ddot x(t)+\frac{\gamma(t)}{\lambda(t)}\dot x(t)+x(t)-x^*$ converges strongly to $0$ as $t\rightarrow + \infty$.
By using again the boundedness of $\lambda$ and $\gamma$ and assumption (ii) we obtain that $x(t)$ converges strongly to $x^*$ as $t\rightarrow+\infty$. 

Finally, suppose that $B$ is uniformly monotone with corresponding function $\phi_B:[0,+\infty) \rightarrow [0,+\infty]$, which is 
increasing  and vanishes only at $0$. The conclusion follows by letting $t$ in the inequality
$$\<x(t)-x^*,B(x(t))-Bx^*\>\geq \phi_B(\|x(t)-x^*\|) \ \forall t \in [0, +\infty)$$ 
converge to $+\infty$ and by using that $x$ is bounded and $\lim_{t\rightarrow+\infty} (B(x(t)-Bx^*)=0$.  
\end{proof}

\begin{remark}\label{eta=2beta} We would like to emphasize the fact that the statements in Theorem \ref{fb-dyn} remain valid also  for $\eta:=2\beta$. 
Indeed, in this case the cocoercivity of $B$ implies that $\id-\eta B$ is nonexpansive, hence the operator $R = J_{\eta A}\circ(\id -\eta B)$ used 
in the proof is nonexpansive, too, and so the statements in (i)-(iii) follow from Corollary \ref{conv-th-nonexp}. Furthermore, the proof of the statements (iv) and (v) 
can be repeated also for $\eta=2\beta$.  
\end{remark}

In the remaining of this section we turn our attention to optimization problems of the form
\begin{equation*}
\min_{x \in {\cal H}} f(x) + g(x),
\end{equation*}
where $f:{\cal H}\rightarrow\R\cup\{+\infty\}$ is a proper, convex and lower semicontinuous function and $g:{\cal H}\rightarrow \R$ is a convex and 
(Fr\'{e}chet) differentiable function with $\frac{1}{\beta}$-Lipschitz continuous gradient for $\beta > 0$.

We recall some standard notations and facts in convex analysis. For a proper, convex and 
lower semicontinuous function $f:{\cal H}\rightarrow\R\cup\{+\infty\}$, its (convex) subdifferential at $x\in {\cal H}$ is defined as
$$\partial f(x)=\{u\in {\cal H}:f(y)\geq f(x)+\<u,y-x\> \ \forall y\in {\cal H}\}.$$ When seen as a set-valued mapping, it is a 
maximally monotone operator (see \cite{rock}) and its resolvent is given by $J_{\eta \partial f}=\prox_{\eta f}$ (see \cite{bauschke-book}),
where $\prox_{\eta f}:{\cal H}\rightarrow {\cal H}$,
\begin{equation}\label{prox-def}\prox\nolimits_{\eta f}(x)=\argmin_{y\in {\cal H}}\left \{f(y)+\frac{1}{2\eta}\|y-x\|^2\right\},
\end{equation}
denotes the proximal point operator of $f$ and $\eta>0$. According to \cite[Definition 10.5]{bauschke-book}, $f$ is said to be uniformly convex with modulus function
$\phi:[0,+\infty)\rightarrow[0,+\infty]$, if $\phi$ is increasing, vanishes only at $0$ and fulfills
$f(\alpha x+(1-\alpha)y)+\alpha(1-\alpha)\phi(\|x-y\|)\leq \alpha f(x)+(1-\alpha)f(y)$ for all $\alpha\in(0,1)$ and 
$x,y\in \dom f:=\{x\in {\cal H}:f(x)<+\infty\}$. Notice that if this inequality holds for $\phi=(\nu/2)|\cdot|^2$ for $\nu>0$, then $f$ is said to be 
$\nu$-strongly convex. 

In the following statement we approach the minimizers of $f+g$ via the second order dynamical system
\begin{equation}\label{dyn-syst-fb-opt}\left\{
\begin{array}{ll}
\ddot x(t) + \gamma(t) \dot x(t) + \lambda(t)\left[x(t)-\prox_{\eta f}\Big(x(t)-\eta \nabla g(x(t))\Big)\right]=0\\
x(0)=u_0, \dot x(0)=v_0.
\end{array}\right.\end{equation}

\begin{corollary}\label{fb-dyn-opt} Let $f:{\cal H}\rightarrow\R\cup\{+\infty\}$ by a proper, convex and 
lower semicontinuous function and $g:{\cal H}\rightarrow \R$ be a convex and (Fr\'{e}chet) differentiable function with  $\frac{1}{\beta}$-Lipschitz continuous gradient for $\beta >0$
such that $\argmin_{x\in {\cal H}}\{f(x)+g(x)\}\neq\emptyset$. Let $\eta\in(0,2\beta]$ and set $\delta:=\frac{4\beta-\eta}{2\beta}$. Let 
$\lambda, \gamma :[0,+\infty)\rightarrow(0,+\infty)$ be functions fulfilling {\rm (A4)}, $u_0,v_0\in {\cal H}$ and 
$x:[0,+\infty)\rightarrow {\cal H}$ be the unique strong global solution of \eqref{dyn-syst-fb-opt}. 
Then the following statements are true: 

(i) the trajectory $x$ is bounded and $\dot x,\ddot x,\big(\id -\prox_{\eta f}\circ(\id -\eta \nabla g)\big)x\in L^2([0,+\infty); {\cal H})$;  

(ii) $\lim_{t\rightarrow+\infty}\dot x(t)=\lim_{t\rightarrow+\infty}\ddot x(t)=
\lim_{t\rightarrow+\infty}\big(\id -\prox_{\eta f}\circ(\id -\eta \nabla g)\big)(x(t))=0$; 

(iii) $x(t)$ converges weakly to a minimizer of $f+g$  as $t\rightarrow+\infty$;

(iv) if $x^*$ is a minimizer of $f+g$, then $\nabla g(x(\cdot))-\nabla g (x^*)\in L^2([0,+\infty); {\cal H})$, 
$\lim_{t\rightarrow+\infty}$ $\nabla g(x(t))=\nabla g(x^*)$ and $\nabla g$ is constant on $\argmin_{x\in {\cal H}}\{f(x)+g(x)\}$; 

(v) if $f$ or $g$ is uniformly convex, then $x(t)$ converges strongly to the unique minimizer of $f+g$ as $t\rightarrow+\infty$.
\end{corollary}

\begin{proof} The statements are direct consequences of the corresponding ones from Theorem \ref{fb-dyn} (see also Remark \ref{eta=2beta}), by choosing
$A:=\partial f$ and $B:=\nabla g$, by taking into account that
$$\zer(\partial f+\nabla g)=\argmin_{x\in {\cal H}}\{f(x)+g(x)\}.$$ 
For statement (v) we also use the fact that if $f$ is uniformly convex 
with modulus $\phi$, then $\partial f$ is uniformly monotone with modulus $2\phi$ (see \cite[Example 22.3(iii)]{bauschke-book}).  
\end{proof}

\begin{remark}\label{alv-att-opt} Consider again the setting in Remark \ref{alv-att-nonexp}, namely, when 
$\gamma(t)=\gamma > 0$ for every $t\geq 0$cand $\lambda(t)=1$ for every $t \in [0,+\infty)$. 
Furthermore, for $C$ a nonempty, convex, closed subset of ${\cal H}$, let $f=\delta_C$ be the indicator function of $C$, which is defined as being equal to $0$ for $x\in C$ and to $+\infty$, else. 
The dynamical system \eqref{dyn-syst-fb-opt} attached in this setting to the minimization of $g$ over $C$ becomes
\begin{equation}\label{dyn-syst-fb-opt-alv-att}\left\{
\begin{array}{ll}
\ddot x(t) + \gamma \dot x(t) + x(t)-P_C\big(x(t)-\eta \nabla g(x(t))\big)=0\\
x(0)=u_0, \dot x(0)=v_0,
\end{array}\right.\end{equation}
where $P_C$ denotes the projection onto the set $C$. 

The asymptotic convergence of the trajectories of \eqref{dyn-syst-fb-opt-alv-att} has been studied in \cite[Theorem 3.1]{att-alv} under the conditions $\gamma^2>2$ and $0<\eta\leq2\beta$.
In this case assumption (A4) trivially holds by choosing $\theta$ such that $0<\theta\leq(\gamma^2-2)/2\leq(\delta\gamma^2-2)/2$. 
Thus, in order to verify (A4) in case  $\lambda(t)=1$ for every $t \in [0,+\infty)$ one needs
to equivalently assume that $\gamma^2>2/\delta$. Since $\delta \geq 1$, this provides a slight improvement over \cite[Theorem 3.1]{att-alv} in what concerns the choice of $\gamma$.
We refer the reader also to \cite{antipin} for an analysis of the convergence rates of trajectories of the dynamical system \eqref{dyn-syst-fb-opt-alv-att} when $g$ is endowed with supplementary properties.
\end{remark}

For the two main convergence statements provided in this section it was essential to choose the step size $\eta$ in the interval $(0, 2\beta]$ (see Theorem \ref{fb-dyn}, Remark \ref{eta=2beta} and Corollary \ref{fb-dyn-opt}).
This, because of the fact that in this way we were able to guarantee for the generated trajectories the existence of the limit $\lim_{t\rightarrow+\infty}\|x(t)-x^*\|^2$, where $x^*$ denotes  
a solution of the problem under investigation. It is interesting to observe that, when dealing with convex optimization 
problems, one can go also beyond this classical restriction concerning the choice of the step size (a similar phenomenon has been 
reported also in \cite[Section 5.2]{abbas-att-arx14}). This is pointed out in the following result, 
which is valid under the assumption
\begin{enumerate}
\item[{\rm (A5)}] $\lambda, \gamma :[0,+\infty)\rightarrow (0,+\infty)$ are locally absolutely continuous and there exists $\theta >0$ such that for almost every $t\in [0, +\infty)$ we have  
\begin{equation}\label{h3''''-g}\dot\gamma(t)\leq 0\leq\dot\lambda(t) \mbox{ and } \frac{\gamma^2(t)}{\lambda(t)}\geq  \eta\theta+\frac{\eta}{\beta}+1.\end{equation}
\end{enumerate}
and for the proof of which we use instead of $\|x(\cdot)-x^*\|^2$ a modified energy functional. 

\begin{corollary}\label{fb-dyn-opt2} Let $f:{\cal H}\rightarrow\R\cup\{+\infty\}$ by a proper, convex and 
lower semicontinuous function and $g:{\cal H}\rightarrow \R$ be a convex and (Fr\'{e}chet) differentiable function with $\frac{1}{\beta}$-Lipschitz continuous gradient for $\beta >0$ such that $\argmin_{x\in {\cal H}}\{f(x)+g(x)\}\neq\emptyset$. 
Let be $\eta>0$, $\lambda, \gamma :[0,+\infty)\rightarrow(0,+\infty)$ be functions fulfilling {\rm (A5)}, 
$u_0,v_0\in {\cal H}$ and $x:[0,+\infty)\rightarrow {\cal H}$ be the unique strong global solution of 
\eqref{dyn-syst-fb-opt}. Then the following statements are true: 

(i) the trajectory $x$ is bounded and $\dot x,\ddot x,\big(\id -\prox_{\eta f}\circ(\id -\eta \nabla g)\big)x\in L^2([0,+\infty); {\cal H})$;  

(ii) $\lim_{t\rightarrow+\infty}\dot x(t)=\lim_{t\rightarrow+\infty}\ddot x(t)=
\lim_{t\rightarrow+\infty}\big(\id -\prox_{\eta f}\circ(\id -\eta \nabla g)\big)(x(t))=0$; 

(iii) $x(t)$ converges weakly to a minimizer of $f+g$  as $t\rightarrow+\infty$;

(iv) if $x^*$ is a minimizer of $f+g$, then $\nabla g(x(\cdot))-\nabla g (x^*)\in L^2([0,+\infty); {\cal H})$, 
$\lim_{t\rightarrow+\infty}$ $\nabla g(x(t))=\nabla g(x^*)$ and $\nabla g$ is constant on $\argmin_{x\in {\cal H}}\{f(x)+g(x)\}$; 

(v) if $f$ or $g$ is uniformly convex, then $x(t)$ converges strongly to the unique minimizer of $f+g$ as $t\rightarrow+\infty$.
\end{corollary}

\begin{proof} Consider an arbitrary element $x^* \in \argmin_{x\in {\cal H}}\{f(x)+g(x)\}=\zer(\partial f+\nabla g)$. 
Similarly to the proof of Theorem \ref{fb-dyn}(iv), we derive for every $t \in [0, +\infty)$ 
(see the first inequality after \eqref{ineq-fb-mon-g}) 
\begin{align} \label{ineq-opt2}
& \beta\|\nabla g(x(t))-\nabla g(x^*)\|^2 \leq \nonumber \\ 
& \frac{1}{\lambda(t)}\Big(\<\ddot x(t),-\nabla g(x(t))+\nabla g(x^*)\> +\gamma(t)\<\dot x(t) , -\nabla g(x(t))+\nabla g(x^*)\>\Big) - \nonumber \\
& \frac{1}{\eta\lambda^2(t)}\|\ddot x(t)+\gamma(t)\dot x(t)\|^2 + \<x(t)-x^* , -\frac{1}{\eta\lambda(t)}\ddot x(t)-\frac{\gamma(t)}{\eta\lambda(t)}\dot x(t)\>.
\end{align}

In what follows we evaluate the right-hand side of the above inequality and introduce to this end the function 
$$q:[0,+\infty)\rightarrow\R, \ q(t)=g(x(t))-g(x^*)-\<\nabla g(x^*),x(t)-x^*\>.$$ 
Due  to the convexity of $g$ one has $$q(t)\geq 0 \ \forall t\geq 0.$$
Further, for every $t \in [0, +\infty)$
$$\dot q(t)=\<\dot x(t),\nabla g(x(t))-\nabla g(x^*)\>,$$
thus
\begin{equation}\label{ineq-du}\gamma(t)\<\dot x(t) , -\nabla g(x(t))+\nabla g(x^*)\>  = -\gamma(t)\dot q(t)= -\frac{d}{dt}(\gamma q)(t)+\dot\gamma(t) q(t)\leq -\frac{d}{dt}(\gamma q)(t)\end{equation}
On the other hand, for every $t \in [0, +\infty)$
$$\ddot q(t)=\<\ddot x(t),\nabla g(x(t))-\nabla g(x^*)\>+\<\dot x(t),\frac{d}{dt}\nabla g(x(t))\>,$$
hence \begin{equation}\label{ineq-ddu}\<\ddot x(t),-\nabla g(x(t))+\nabla g(x^*)\>\leq -\ddot q(t)+\frac{1}{\beta}\|\dot x(t)\|^2.\end{equation}
We have for almost every $t \in [0, +\infty)$ (see also \eqref{d-g-h-x})
\begin{align} \label{ineq-dx-ddx}
\frac{1}{\lambda(t)}\|\ddot x(t)+\gamma(t)\dot x(t)\|^2  = & \ \frac{1}{\lambda(t)}\|\ddot x(t)\|^2+
\frac{\gamma^2(t)}{\lambda(t)}\|\dot x(t)\|^2\nonumber+\\
&\frac{d}{dt}\left(\frac{\gamma(t)}{\lambda(t)}\|\dot x(t)\|^2\right)-
\frac{\dot \gamma(t)\lambda(t)-\gamma(t)\dot\lambda(t)}{\lambda^2(t)}\|\dot x(t)\|^2.\end{align}
Finally, by multiplying \eqref{ineq-opt2} with $\lambda(t)$ and by using \eqref{ineq-du}, \eqref{ineq-ddu}, \eqref{ineq-dx-ddx} and 
\eqref{eq-h-p-g} we obtain after rearranging the terms for almost every $t \in [0, +\infty)$ that
\begin{align*} 
& \ \beta\lambda(t)\|\nabla g(x(t))-\nabla g(x^*)\|^2+\!\!\frac{d}{dt^2}\left(\frac{1}{\eta}h+q\right)
+\!\frac{d}{dt}\left(\gamma(t)\left(\frac{1}{\eta}h+q\right)\right) + \frac{1}{\eta}\frac{d}{dt}\left(\frac{\gamma(t)}{\lambda(t)}\|\dot x(t)\|^2\right) \nonumber\\
& \ +\left(\frac{\gamma^2(t)}{\eta\lambda(t)}+\frac{-\dot \gamma(t)\lambda(t)+\gamma(t)\dot\lambda(t)}{\eta\lambda^2(t)}-\frac{1}{\beta}-\frac{1}{\eta}\right)\|\dot x(t)\|^2\ +\frac{1}{\eta\lambda(t)}\|\ddot x(t)\|^2\leq 0.
\end{align*} 
This relation gives rise via (A5) to 
\begin{align}\label{ineq-g-q} 
& \ \beta\lambda(t)\|\nabla g(x(t))-\nabla g(x^*)\|^2+\!\!\frac{d}{dt^2}\left(\frac{1}{\eta}h+q\right)
+\!\frac{d}{dt}\left(\gamma(t)\left(\frac{1}{\eta}h+q\right)\right)\nonumber\\ 
+ & \  \frac{1}{\eta}\frac{d}{dt}\left(\frac{\gamma(t)}{\lambda(t)}\|\dot x(t)\|^2\right) +
\theta \|\dot x(t)\|^2\ +\frac{1}{\eta\lambda(t)}\|\ddot x(t)\|^2\leq 0,
\end{align} 
for almost every $t \in [0, +\infty)$. This implies that the function 
\begin{equation}\label{decr-f}t\mapsto \frac{d}{dt}\left(\frac{1}{\eta}h+q\right)(t)
+\gamma(t)\left(\frac{1}{\eta}h+q\right)(t)
+\frac{1}{\eta}\left(\frac{\gamma(t)}{\lambda(t)}\|\dot x(t)\|^2\right)\end{equation} is 
monotonically decreasing. Arguing as in the proof of Theorem \ref{conv-th}, by taking into account that $\lambda,\gamma$ have positive upper and lower bounds, it follows that $\frac{1}{\eta}h+q$, $h$, $q$,  
$x,\dot x,\dot h,\dot q$ are bounded and $\dot x,\ddot x, \big(\id -\prox_{\eta f}\circ(\id -\eta \nabla g)\big)x\in L^2([0,+\infty); {\cal H})$. Furthermore, $\lim_{t\rightarrow+\infty}\dot x(t)=0$. Since 
$\frac{d}{dt} \big(\id -\prox_{\eta f}\circ(\id -\eta \nabla g)\big)x\in L^2([0,+\infty); {\cal H})$ 
(see Remark \ref{rem-abs-cont}(b)), we derive 
from Lemma \ref{fejer-cont2} that $\lim_{t\rightarrow+\infty}\big(\id -\prox_{\eta f}\circ(\id -\eta \nabla g)\big)(x(t))=0$. 
As 
$$\ddot x(t)=-\gamma(t)\dot x(t)-\lambda(t)\big(\id -\prox\nolimits_{\eta f}\circ(\id -\eta \nabla g)\big)(x(t))$$
for every $t \in [0,+\infty)$, we obtain that $\lim_{t\rightarrow+\infty}\ddot x(t)=0$. From \eqref{ineq-g-q} it also follows that 
$\nabla g(x(\cdot))-\nabla g (x^*)\in L^2([0,+\infty); {\cal H})$ and, by applying again Lemma \ref{fejer-cont2}, 
it yields $\lim_{t\rightarrow+\infty}\nabla g(x(t))=\nabla g(x^*)$. In this way the statements (i), (ii) and (iv) are shown.

(iii) Since the function in \eqref{decr-f} is monotonically decreasing, from (i), (ii) and (iv) it follows that the limit
$\lim_{t\rightarrow+\infty}\left(\gamma(t)\left(\frac{1}{\eta}h+u\right)(t)\right)$ exists and it is a real number. From 
$\lim_{t\rightarrow+\infty}\gamma(t)\in(0,+\infty)$ we get that  
$\exists \lim_{t\rightarrow+\infty}\left(\frac{1}{\eta}h+u\right)(t)\in\R$. 

Furthermore, since $x^*$ has been chosen as an arbitrary minimizer of $f+g$, we conclude that for all $x^*\in\argmin_{x\in {\cal H}}\{f(x)+g(x)\}$ the limit 
$$\lim_{t\rightarrow+\infty}E(t,x^*)\in\R$$
exists, where $$E(t,x^*)=\frac{1}{2\eta}\|x(t)-x^*\|^2+g(x(t))-g(x^*)-\<\nabla g(x^*),x(t)-x^*\>.$$

In what follows we use a similar technique as in \cite{bolte-2003} (see, also, \cite[Section 5.2]{abbas-att-arx14}). Since $x(\cdot)$ is bounded, it has at least one weak sequential cluster point. 

We prove first that each weak sequential cluster point of $x(\cdot)$ is a minimizer of $f+g$. Let $x^*\in\argmin_{x\in {\cal H}}\{f(x)+g(x)\}$ and $t_n\rightarrow+\infty$ (as $n\rightarrow+\infty$) 
be such that $(x(t_n))_{n\in\N}$ converges weakly to $\ol x$. Since $(x(t_n),\nabla g(x(t_n)))\in\gr(\nabla g)$, $\lim_{n \rightarrow+\infty}\nabla g(x(t_n))=\nabla g(x^*)$ and 
$\gr(\nabla g)$ is sequentially closed in the weak-strong topology, we obtain $\nabla g(\ol x)=\nabla g(x^*)$. 

From \eqref{conseq-def-res-g} written for  $t=t_n$, $A=\partial f$ and $B=\nabla g$, by letting $n$ converge to $+\infty$ and by using that $\gr(\partial f)$ is sequentially closed in the weak-strong topology, 
we obtain $-\nabla g(x^*)\in\partial f(\ol x)$. This, combined with $\nabla g(\ol x)=\nabla g(x^*)$, delivers $-\nabla g(\ol x)\in\partial f(\ol x)$, 
hence $\ol x\in\zer(\partial f+\nabla g)=\argmin_{x\in {\cal H}}\{f(x)+g(x)\}$.

Next we show that $x(\cdot)$ has at most one weak sequential cluster point, fact which guarantees that it has exactly one weak sequential cluster point.
This implies the weak convergence of the trajectory to a minimizer of $f+g$. 

Let $x_1^*,x_2^*$ be two weak sequential cluster points of $x(\cdot)$. This means that there exist $t_n\rightarrow+\infty$ (as $n\rightarrow+\infty$)
and $t_n'\rightarrow+\infty$ (as $n\rightarrow+\infty$) such that $(x(t_n))_{n\in\N}$ converges weakly to $x_1^*$ (as $n\rightarrow+\infty$) and 
$(x(t_n'))_{n\in\N}$ converges weakly to $x_2^*$ (as $n\rightarrow+\infty$). Since $x_1^*, x_2^*\in \argmin_{x\in {\cal H}}\{f(x)+g(x)\}$, 
we have $\lim_{t\rightarrow+\infty}E(t,x_1^*)\in\R$ and $\lim_{t\rightarrow+\infty}E(t,x_2^*)\in\R$, hence 
$\exists \lim_{t\rightarrow+\infty}(E(t,x_1^*)-E(t,x_2^*))\in\R.$ We obtain 
$$\exists\lim_{t\rightarrow+\infty}\left(\frac{1}{\eta}\<x(t),x_2^*-x_1^*\>+\<\nabla g(x_2^*)-\nabla g(x_1^*),x(t)\>\right)\in\R,$$
which, when expressed by means of the sequences $(t_n)_{n\in\N}$ and $(t_n')_{n\in\N}$, leads to
$$\frac{1}{\eta}\<x_1^*,x_2^*-x_1^*\>+\<\nabla g(x_2^*)-\nabla g(x_1^*),x_1^*\>=\frac{1}{\eta}\<x_2^*,x_2^*-x_1^*\>+\<\nabla g(x_2^*)-\nabla g(x_1^*),x_2^*\>.$$
This is the same with
$$\frac{1}{\eta}\|x_1^*-x_2^*\|^2+\<\nabla g(x_2^*)-\nabla g(x_1^*) , x_2^*-x_1^*\>=0$$
and by the monotonicity of $\nabla g$ we conclude that $x_1^*=x_2^*$. 

(v) The proof of this statement follows in analogy to the one of the corresponding statement of Theorem \ref{fb-dyn}(v) written for $A=\partial f$ and $B=\nabla g$. 
\end{proof}

\begin{remark}\label{opt2} When $\gamma(t)=\gamma >0$ for every $t\geq 0$ and $\lambda(t)=1$ for every $t \in [0,+\infty)$, 
the second inequality in \eqref{h3''''-g} is verified if and only if $\gamma^2>\frac{\eta}{\beta}+1$. 
In other words, (A5) allows in this particular setting a more relaxed choice
for the parameters $\gamma,\eta$ and $\beta$, beyond the standard assumptions $0<\eta\leq 2\beta$ and $\gamma^2>2$ 
considered in \cite{att-alv}.  
\end{remark}

\begin{remark}\label{discr} The explicit discretization of \eqref{dyn-syst-fb-opt} with respect to the time variable $t$, with step size $h_n > 0$, relaxation variable $\lambda_n >0$,
damping variable $\gamma_n > 0$ and initial points $x_0:= u_0$ and $x_1:= v_0$ yields the following iterative scheme
$$\frac{x_{n+1} - 2x_n + x_{n-1}}{h_n^2} + \gamma_n \frac{x_{n+1}-x_n}{h_n} = \lambda_n\left[\prox\nolimits_{\eta f}\Big(x_n-\eta \nabla g(x_n)\Big) - x_n\right] \ \forall n \geq 1.$$
For $h_n=1$ this becomes 
\begin{align*}
x_{n+1} = \left(1- \frac{\lambda_n}{1+\gamma_n} \right)x_n + \frac{\lambda_n}{1+\gamma_n} \prox\nolimits_{\eta f}\Big(x_n-\eta \nabla g(x_n)\Big) + \frac{\lambda_n}{1+\gamma_n}(x_{n} - x_{n-1})
\ \forall n \geq 1,
\end{align*}
which is a relaxed forward-backward algorithm for minimizing $f+g$ with inertial effects. For more on inertial-type forward-backward algorithms we refer the reader to \cite{moudafi-oliny2003}.
\end{remark}

In the following we provide a rate for the convergence for a  convex and  (Fr\'{e}chet) differentiable function
 $g:{\cal H}\rightarrow\R$ with Lipschitz continuous gradient  to its minimum value along the ergodic trajectory generated by 
\begin{equation}\label{dyn-syst-grad}\left\{
\begin{array}{ll}
\ddot x(t) + \gamma(t) \dot x(t) + \lambda(t)\nabla g(x(t))=0\\
x(0)=u_0, \dot x(0)=v_0.
\end{array}\right.\end{equation}
To this end we make the following assumption:
\begin{enumerate}
\item[{\rm (A6)}] $\lambda :[0,+\infty)\rightarrow (0,+\infty)$ is locally absolutely continuous, $\gamma :[0,+\infty)\rightarrow (0,+\infty)$ is twice differentiable and there exists $\zeta >0$ such that for almost every $t\in [0, +\infty)$ we have  
\begin{equation}\label{h3'''''-g} 0 < \zeta \leq \gamma(t)\lambda(t)-\dot\lambda(t), \ \dot \gamma(t)\leq 0 
\mbox{ and }2\dot\gamma(t)\gamma(t)-\ddot \gamma(t)\leq 0.\end{equation}
\end{enumerate}
Let us mention that the following result is in the spirit of a convergence rate statement recently given in \cite[Theorem 1]{gfj} for the objective function values on a sequence iteratively generated by an inertial gradient-type algorithm.

\begin{theorem}\label{rate-g} Let $g:{\cal H}\rightarrow\R$ be a convex and (Fr\'{e}chet) differentiable function with $\frac{1}{\beta}$-Lipschitz 
continuous gradient for $\beta > 0$ such that $\argmin_{x\in{\cal H}} g(x)\neq\emptyset$. Let 
$\lambda, \gamma :[0,+\infty)\rightarrow(0,+\infty)$ be functions fulfilling {\rm (A6)} $u_0,v_0\in {\cal H}$ and 
$x:[0,+\infty)\rightarrow {\cal H}$ be the unique strong global solution of
\eqref{dyn-syst-grad}.

Then for every minimizer $x^*$ of $g$ and every $T> 0$ it holds
\begin{align*}
0\leq &  g\left(\frac{1}{T}\int_0^Tx(t)dt\right)-g(x^*)\leq \\
& \frac{1}{2\zeta T}\left[\|v_0+\gamma(0)(u_0-x^*)\|^2+\left(\frac{\lambda(0)}{\beta}-\dot \gamma(0)\right)\|u_0-x^*\|^2\right].
\end{align*}
\end{theorem}

\begin{proof} By using \eqref{dyn-syst-grad}, the convexity of $g$ and (A6) we get for almost every $t \in [0,+\infty)$
\begin{align*}
& \frac{d}{dt}\left(\frac{1}{2}\|\dot x(t)+\gamma(t)(x(t)-x^*)\|^2+\lambda(t)g(x(t))-\frac{\dot\gamma(t)}{2}\|x(t)-x^*\|^2\right)\\
= & \<\ddot x(t)+\dot\gamma(t)(x(t)-x^*)+\gamma(t)\dot x(t),\dot x(t)+\gamma(t)(x(t)-x^*)\>\\
& -\frac{\ddot\gamma(t)}{2}\|x(t)-x^*\|^2-
\dot\gamma(t)\<\dot x(t),x(t)-x^*\>+\dot \lambda(t)g(x(t))+\lambda(t)\<\dot x(t),\nabla g(x(t))\>\\
= &-\gamma(t)\lambda(t)\<\nabla g(x(t)) , x(t)-x^*\>+\dot\lambda(t)g(x(t))+\left(\dot\gamma(t)\gamma(t)-\frac{\ddot\gamma(t)}{2}\right)\|x(t)-x^*\|^2\\
\leq & -\gamma(t)\lambda(t)\<\nabla g(x(t)) , x(t)-x^*\>+\dot\lambda(t)g(x(t))\\
\leq & (\dot\lambda(t)-\gamma(t)\lambda(t))(g(x(t))-g(x^*))+\dot\lambda(t)g(x^*)\\
\leq & -\zeta(g(x(t))-g(x^*))+\dot\lambda(t)g(x^*).
\end{align*}
We obtain after integration 
\begin{align*}
\frac{1}{2}\|\dot x(T)+\gamma(T)(x(T)-x^*)\|^2+\lambda(T) g(x(T)) - 
\frac{\dot\gamma(T)}{2}\|x(T)-x^*\|^2 & \\
- \frac{1}{2}\|\dot x(0)+\gamma(0)(x(0)-x^*)\|^2-\lambda(0) g(x(0)) + 
\frac{\dot\gamma(0)}{2}\|x(0)-x^*\|^2 & \\
+ \zeta\int_0^T(g(x(t))-g(x^*))dt & \leq (\lambda(T)-\lambda(0))g(x^*).
\end{align*}
Be neglecting the nonnegative terms in the left-hand side of the  inequality above and by using that $g(x(T))\geq g(x^*)$, it yields
$$\zeta\int_0^T(g(x(t))-g(x^*))dt\leq\frac{1}{2}\|v_0+\gamma(0)(u_0-x^*)\|^2-\frac{\dot\gamma(0)}{2}\|u_0-x^*\|^2+ \lambda(0)(g(u_0)-g(x^*)).$$
The conclusion follows by using
$$g(u_0)-g(x^*)\leq \frac{1}{2\beta}\|u_0-x^*\|^2,$$
which is a consequence of the descent lemma (see \cite[Lemma 1.2.3]{nes} and notice that $\nabla g(x^*) = 0$), and the inequality
$$g\left(\frac{1}{T}\int_0^Tx(t)dt\right)-g(x^*)\leq \frac{1}{T}\int_0^T(g(x(t))-g(x^*))dt,$$
which holds since $g$ is convex.
\end{proof}

\begin{remark}\label{it-func-val} Under assumption (A6), we obtain in the above theorem (only) the convergence of the function 
$g$ along the ergodic trajectory to a global minimum value.
If one is interested also in the (weak) convergence of the trajectory to a minimizer of $g$, this follows via 
Theorem \ref{conv-th} when $\lambda,\gamma$ are assumed to fulfill (A1) (notice that if $x$ converges weakly to a
minimizer of $g$, then from the Cesaro-Stolz Theorem one also obtains the weak convergence of the ergodic 
trajectory $T\mapsto\frac{1}{T}\int_0^Tx(t)dt$ to the same minimizer). 

For $a,a',\rho,\rho'\geq 0$ and $b,b'>0$ fulfilling the inequalities ${b'}^2b>\frac{1}{\beta}$ and $0\leq\rho\leq b'$ one can prove that the functions
$$\lambda(t)=\frac{1}{ae^{-\rho t}+b} \mbox{ and } \gamma(t)=a'e^{-\rho't}+b',$$
verify assumption (A1) in Theorem \ref{conv-th} for $0<\theta\leq b'^2b\beta-1$ 
and assumption (A6) in Theorem \ref{rate-g} for $0<\zeta\leq  \frac{bb'}{(a+b)^2}$. Hence, for this choice of the relaxation and 
damping functions, both convergence of both the objective function $g$ along 
the ergodic trajectory to its global minimum value and (weak) convergence of the trajectory 
to a minimizer of $g$ are guaranteed. 
\end{remark}

{\bf Acknowledgements.} The authors are thankful to the handling editor and two anonymous reviewers for comments and remarks which substantially improved the quality of the paper.

\end{document}